\documentclass[preprint]{elsarticle}

\usepackage{amsmath,amstext, amsthm,amsfonts,amssymb,amsthm}

\biboptions{sort&compress}

\newtheorem{thm}{Theorem}[section]
\newtheorem{cor}[thm]{Corollary}
\newtheorem{lem}[thm]{Lemma}
\newtheorem{prop}[thm]{Proposition}
\theoremstyle{definition}
\newtheorem{defn}[thm]{Definition}
\theoremstyle{remark}
\newtheorem{rem}[thm]{Remark}
\newtheorem{exa}[thm]{Example}
\numberwithin{equation}{section}

\newcommand{\norm}[1]{\left\Vert#1\right\Vert}

\newcommand{\set}[1]{\left\{#1\right\}}

\newcommand{\RD}{D_{q,0^+}^{\alpha}}
\newcommand{\CP}{{}^cD_{q,a^+}}
\newcommand{\CM}{{}^cD_{q,{a}^-}}
\newcommand{\CZ}{{}^cD_{q,0^+}}

\newcommand{\To}{\longrightarrow}

\newcommand{\be}{\begin{equation}}
\newcommand{\ee}{\end{equation}}
\newcommand{\non}{\nonumber}
\newcommand{\bea}{\begin{eqnarray}}
\newcommand{\eea}{\end{eqnarray}}
\newcommand{\E}{{}_0E_{a}^{\alpha}}
\newcommand{\V}{{}^c\text{Var}(0,a)}

\newcommand{\ds}{\displaystyle}

\newcommand*{\Scale}[2][4]{\scalebox{#1}{$#2$}}%

\journal{}

\begin{document}

\begin{frontmatter}



\title{Variational methods for  fractional $q$-Sturm--Liouville Problems}


\author{Zeinab S.I. Mansour}

\address{Department of Mathematics, Faculty of Science, King Saud University, Riyadh}
 \ead{zsmansour@ksu.edu.sa, zeinabs98@hotmail.com}

\begin{abstract}
In this paper, we formulate a regular $q$-fractional Sturm--Liouville problem (qFSLP) which  includes  the left-sided Riemann--Liouville and  the right-sided Caputo $q$-fractional  derivatives  of the same  order $\alpha$, $\alpha\in (0,1)$.     We  introduce the essential  $q$-fractional  variational analysis  needed in proving  the existence of a countable set of real eigenvalues and associated  orthogonal eigenfunctions for  the regular qFSLP when  $\alpha>1/2$ asscociated with the boundary condition  $y(0)=y(a)=0$. A criteria for the first eigenvalue is proved. Examples are included.  These results  are a generalization of  the  integer regular $q$-Sturm--Liouville problem introduced by Annaby and Mansour in\cite{AM1}.
\end{abstract}


\end{frontmatter}

\bigskip
\section{\bf{Introduction}}

In the joint paper of Sturm and Liouville~\cite{Sturm-Liouville},  they studied the problem
\be \label{SLE} -\frac{d}{dx}\left(p\frac{dy}{dx}\right)+r(x)y(x)=\lambda w y(x),\quad x\in [a,b], \ee
with certain boundary conditions at $a$ and $b$. Here, the functions $p$, $w$ are positive on $[a,b]$ and $r$ is a real valued function on $[a,b]$. They proved the existence
of non-zero solutions (eigenfunctions) only for special values of the parameter $\lambda$ which is called eigenvalues.
For a comprehensive study for the contribution of Sturm and Liouville to the theory, see~\cite{STL}. Recently, many mathematicians were interested in a fractional version of
\eqref{SLE}, i.e. when the derivative is replaced by a fractional derivative like  Riemann--Liouville derivative or Caputo derivative, see~\cite{KOM,Klimek-Agrawal,FSL1,FSL2,FSL3,FSL4}. Iterative methods, variational method, and the fixed point  theory are three  different approaches used in  proving the existence and uniqueness of solutions
of Sturm--Liouville problems, c.f.~\cite{GFbook,Zettl,STL}.  The calculus of variations
has recently developed to calculate extremum of functional contains fractional derivatives, which is called fractional calculus of variations, see for example~\cite{Almedia-Malinowska-Torres, Almedia-Torres1, Almedia-Torres2,Almedia-Torres3,Agarwal2002,Agrawal2007,Agrawal2007-2,Agrawal2006}. In~\cite{KOM},
Klimek et al. applied the methods of fractional variational calculus to prove the existence of a countable set of orthogonal solutions and corresponding eigenvalues.
In~\cite{AM1} Annaby and Mansour introduced a $q$-version of \eqref{SLE}, i.e., when the derivative is replaced by Jackson $q$-derivative. Their results are applied and developed in different aspects, for example,
see~\cite{Lavagno,Abreu005,AnBuIs,AMI,Nemri-Fitouhi,Abreu007}.
 Throughout this paper $q$ is a positive number less than 1. The set of non negative integers is denoted by $\mathbb{N}_0$, and the set of positive integers is denoted by $\mathbb{N}$. For $t>0$,
\[A_{q,t}:=\set{tq^n\;:\;n\in\mathbb{N}_0},\;\;\;A_{q,t}^*:=A_{q,t}\cup\set{0},\]
and \[\mathcal{A}_{q,t}:=\set{\pm tq^n\;:\;n\in\mathbb{N}_0}.\]
When $t=1$, we simply use $A_{q}$, $A_q^*$, and $\mathcal{A}_q$ to denote $A_{q,1}$, $A_{q,1}^*$, and $\mathcal{A}_{q,1}$, respectively.    We follow~\cite{GR} for the definitions and notations of the $q$-shifted factorial, the $q$-gamma and $q$-beta functions, the basic hypergeometric series, and
Jackson $q$-difference operator and integrals. A set $A$ is called  a $q$-geometric set if $qx\in A$ whenever $x\in A$.
Let $X$ be a $q$-geometric set containing zero. A function $f$ defined on $X$ is called $q$-regular at zero if
\[\lim_{n\to\infty} f(xq^n)=f(0)\quad\mbox{for all}\;x\in X.\]
  Let $C(X)$  denote the space of all $q$-regular at zero functions defined  on
$X$ with values in $\mathbb{R}$. $C(X)$ associated with the norm  function
\[\norm{f} = \sup\set{|f(xq^n)|\; :\; x\in X,\; n\in\mathbb{N}_0},\]
is a normed space.
The $q$-integration by parts rule~\cite{AMbook} is
\be\label{qIR}\int_{a}^{b}f(x)D_qg(x)=f(x)g(x)|_{a}^{b}+\int_{a}^{b}D_qf(x)g(qx)\,d_qx,\; a,b\in X,\ee
and $f,\,g$ are $q$-regular at zero functions.

For $p>0$, and   $Y$ is $A_{q,t}$ or $A_{q,t}^*$, the space $L_q^p(Y)$ is the normed space
of all functions defined on $Y$ such that
\[\norm{f}_p:=\left(\int_{0}^{t}|f(u)|^p\,d_qu\right)^{1/p}<\infty. \]
If $p=2$, then $L_q^2(Y)$  associated with the inner product
\be\label{IP} \left<f,g\right>:=\int_{0}^{t}f(u)\overline{g(u)}\,d_qu\ee
is a Hilbert space. By a weighted $L_q^2(Y, w)$ space is the space of all functions  $f$ defined on $Y$
such that
\[\int_{0}^{t}|f(u)|^2 w(u)\,d_qu<\infty,\]
where $w$ is a positive function defined on $Y$. $L_q^2(Y, w)$ associated with the inner product
\[ \left<f,g\right>:=\int_{0}^{t}f(u)\overline{g(u)}w(u)\,d_qu\]
is a Hilbert space.
The space of all $q$-absolutely functions on $A_{q,t}^*$ is denoted by $\mathcal{A}C_q(A_{q,t}^*)$ and defined as the space of all $q$-regular
at zero functions $f$
satisfying \[\sum_{j=0}^{\infty}|f(uq^j)-f(uq^{j+1})|\leq K\;\mbox{for all}\; u\in A_{q,t}^{*}, \]
and $K$ is a constant depending on the function $f$, c.f.~\cite[Definition 4.3.1]{AMbook}. I.e.
\[\mathcal{A}C_q(A_{q,t}^*)\subseteq C_q(A_{q,t}^*).\]
The  space $\mathcal{A}C_q^{(n)}(A_{q,t}^*)$  ($n\in\mathbb{N}$) is the space
of all functions defined on $X$ such that  $f,\, D_qf,\,\ldots,\, D_q^{n-1}f$ are $q$-regular at zero and $D_q^{n-1}f\in \mathcal{A}C_q(A_{q,t}^*)$, c.f.~\cite[Definition 4.3.2]{AMbook}.
Also it is proved in~\cite[Theorem 4.6]{AMbook} that a function $f\in \mathcal{A}C_q^{(n)}(A_{q,t}^*)$ if and only if there exists a function $\phi\in L_q^1(A_{q,t}^*)$
such that
\[f(x)=\sum_{k=0}^{n-1}\frac{D_q^kf(0)}{\Gamma_q(k+1)}x^k+\frac{x^{n-1}}{\Gamma_q(n)}\int_{0}^{x} (qu/x;q)_{n-1}\phi(u)\,d_qu,\;x\in A_{q,t}^*.\]
In particular, $f\in \mathcal{A}C(A_{q,t}^*)$ if and only if $f$ is $q$-regular at zero such that $D_qf\in L_q^1(A_{q,t}^*)$.
It is worth noting that in~\cite{AMbook}, all the definitions and results we have just mentioned are defined and proved for functions defined on the interval $[0,a]$
instead of $A_{q,t}^*$.
In~\cite{Zeinab_FSL1}, Mansour studied the problem
\be\label{EVP1}
D_{q,a^-}^{\alpha}p(x)\CZ^{\alpha}y(x) +\left(r(x)-\lambda w_{\alpha}(x)\right)y(x)=0, \quad x\in A_{q,a}^*,
\ee
where $p(x)\neq 0$ and $w_{\alpha}>0$ for all $x\in A_{q,a}^*$, $p,r,\,w_{\alpha}$ are real valued functions defined in $A_{q,a}^*$
and the associated  boundary conditions are
\be\label{BC1}
c_1 y(0)+c_2 \left[I_{q,a^-}^{1-\alpha}\,p\CZ^{\alpha}y\right](0)=0,
\ee
\be\label{BC2}
d_1 y(a)+d_2\left[ I_{q,a^-}^{1-\alpha}\,p\CZ^{\alpha}y\right](\frac{a}{q})=0,\ee
with $c_1^2+c_2^2\neq 0$ and $d_1^2+d_2^2\neq 0$.
  it is proved that the eigenvalues are real and the eigenfunctions associated to different eigenvalues are orthogonal in the   Hilbert space $L_q^2(A_{q,a}^*, w_{\alpha})$. A sufficient  condition on
  the parameter $\lambda$ to guarantee the existence  and uniqueness of the solution is introduced by using the fixed point theorem,  also a  condition is imposed  on the domain of the problem in order to prove the existence and uniqueness of  solution
for any $\lambda$.
This paper is organized as follows. Section 2 is on the $q$-fractional operators and their properties which we need in the sequel.
Cardoso~\cite{Cardoso011} introduced basic Fourier series  for functions defined on a $q$-linear grid of the form $\set{\pm q^n:\; n\in\mathbb{N}_0}\cup\set{0}$.
In Section 3,  we reformulate  Cardoso's results for functions defined on a $q$-linear grid of the form $\set{\pm aq^n:\;n\in\mathbb{N}_0}\cup\set{0}$. In Section 4, we introduce a fractional $q$-analogue for Euler--Lagrange
equations for functionals defined in terms of Jackson $q$-integration and the integrand contains  the left sided Caputo fractional $q$-derivative. We also introduce a fractional $q$-isoperimetric problem. In Section 5,  we use the variational  $q$-calculus developed in Section 4 to prove the existence of a countable number of eigenvalues and orthogonal eigenfunctions
for the fractional $q$-Sturm--Liouville problem with the boundary condition $y(0)=y(a)=0$. We also define the Rayleigh quotient and prove a   criteria for the smallest eigenvalue.

\section{Fractional $q$-Calculus}

This section includes the definitions and properties of  the left sided and right sided Riemann--Liouville $q$-fractional operators which we need in our investigations.

The  left   sided  Riemann--Liouville $q$-fractional operator is defined by
\be\label{LS:OP}
I_{q,a^+}^{\alpha}f(x)=\dfrac{x^{\alpha-1}}{\Gamma_q(\alpha)}\int_{a}^{x}(qt/x;q)_{\alpha-1}f(t)\,d_qt.
\ee
 This  definition   is introduced by Agarwal in~\cite{Agr} when  $a=0$ and by Rajkovi\'{c} et.al~\cite{Rajkovic09} for $a\neq 0$.
 The  right sided Riemann--Liouville $q$-fractional operator  by
\be\label{RS:OP}
I_{q,b^-}^{\alpha}f(x)=\frac{1}{\Gamma_q(\alpha)}\int_{qx}^{b}t^{\alpha-1}(qx/t;q)_{\alpha-1}f(t)\,d_qt,
\ee
see~\cite{Zeinab_FSL1}.
The left sided Riemann--Liouville $q$-fractional operator satisfies the semigroup property
\[
I_{q,a^+}^\alpha I_{q,a^+}^{\beta}f(x)=I_{q,a^+}^{\alpha+\beta}f(x).
\]
The case $a=0$ is proved in~\cite{Agr} and   the case $a>0$ is proved  in~\cite{Rajkovic09}.

The right sided Riemann--Liouville $q$-fractional operator satisfies the semigroup property~\cite{Zeinab_FSL1}
\be
 I_{q,b^-}^\alpha I_{q,{{b}}^-}^{\beta}f(x)=I_{q,b^-}^{\alpha+\beta}f(x),\;x\in A_{q,b}^*,
\ee
for any function defined on $A_{q,b}$ and for any values of  $\alpha$ and $\beta$.

For  $\alpha>0$ and  $\ulcorner\alpha\urcorner=m$,   the left and right side  Riemann--Liouville fractional $q$-derivatives of order $\alpha$ are defined by
\[
D_{q,a^+}^\alpha f(x):=D_q^mI_{q,a^+}^{m-\alpha}f(x),\; D_{q,b^-}^\alpha f(x):=\left(\frac{-1}{q}\right)^mD_{q^{-1}}^mI_{q,b^-}^{m-\alpha}f(x),
\]
the left and right sided Caputo fractional $q$-derivatives of order $\alpha$ are defined by
\[
\CP^{\alpha}f(x):=I_{q,a^+}^{m-\alpha}D_q^mf(x),\quad{}^cD_{q,b^-}^\alpha:=\left(\frac{-1}{q}\right)^mI_{q,b^-}^{m-\alpha}D_{q^{-1}}^mf(x).
\]
see~\cite{Zeinab_FSL1}.
From now on, we shall consider left sided Riemann--Liouville and Caputo fractional $q$-derivatives when the lower point $a=0$ and right sided Riemann--Liouville and Caputo fractional $q$-derivatives when $b=a$.
According to~\cite[pp. 124,\,148]{AMbook}, $D_{q,0^+}^{\alpha}f(x)$ exists if
\[f\in L_q^1(A_{q,a}^*)\;\;\mbox{such that}\;\;I_{q,0^+}^{m-\alpha}f\in \mathcal{A}C_q^{(m)}(A_{q,a}^*),\]
and $\CP^{\alpha}f$ exists if
\[f \in \mathcal{A}C_q^{(m)}(A_{q,a}^*).\]

The following proposition  is proved in~\cite{Zeinab_FSL1}
\begin{prop}

\begin{enumerate}Let $\alpha\in(0,1)$.

\item[(i)]If $f\in L_q^1(A_{q,a}^*)$  such that $I_{q,0^+}^{\alpha}f\in \mathcal{A}C_q(A_{q,a}^*)$ then

\be\label{CR3}\CZ^{\alpha}I_{q,0^+}^{\alpha}f(x)=f(x)-\dfrac{I_{q,0^+}^{\alpha}f(0)}{\Gamma_q(1-\alpha)}x^{-\alpha}.
\ee
Moreover, if $f$ is bounded on $A_{q,a}^*$ then
\be\label{CR3-}\CZ^{\alpha}I_{q,0^+}^{\alpha}f(x)=f(x).\ee
\item[(ii)] For any function $f$ defined on $A_{q,a}^*$
\be\label{CR3-1}
 \CM^{\alpha}I_{q,a^-}^{\alpha}f(x)=f(x)-\frac{a^{-\alpha}}{\Gamma_q(1-\alpha)}(qx/a;q)_{-\alpha}\left(I_{q,a^-}^{\alpha}f\right)(\frac{a}{q}).\ee

\item[(iii)] If $f\in L_q^1(A_{q,a})$ then
\be\label{CR1}
D_{q,0^+}^{\alpha}I_{q,0^+}^{\alpha}f(x)=f(x).
\ee

\item[(iv)] For any function $f$ defined on $A_{q,a}^*$
\be\label{CR2}
D_{q,a^-}^{\alpha}I_{q,a^-}^{\alpha}f(x)=f(x).
\ee
\item[(v)] If $f\in\mathcal{A}C_q(A_{q,a}^*))$  then
\be\label{CR4}
I_{q,0^+}^{\alpha}\CZ^{\alpha}f(x)=f(x)-f(0).
\ee
\item[(vi)] If $f$ is a function defined on $A_{q,a}^*$ then
\be\label{CR5}
I_{q,a^-}^{\alpha}D_{q,a^-}^{\alpha}f(x)=f(x)-\frac{a^{\alpha-1}}{\Gamma_q(\alpha)}(qx/a;q)_{\alpha-1}\left(I_{q,a^-}^{1-\alpha}f\right)(\frac{a}{q}).
\ee
\item[(v)] If $f$ is defined on $[0,a]$ such that $D_qf$ is continuous on $[0,a]$ then
\be\label{CR6}
\CZ^{\alpha}f(x)=D_{q,0^+}^{\alpha}\left[f(x)-f(0)\right].
\ee
\end{enumerate}
\end{prop}

Set $X=A_{q,a}$ or $A_{q,a}^*$. Then
\[C(X)\subseteq L_q^2(X)\subseteq L_q^1(X).\]
Moreover, if $f\in C(X)$ then
\[\norm{f}_1\leq \sqrt{a}\norm{f}_2\leq a \norm{f}.\]
We have also the following inequalities:
\begin{enumerate}
\item  If $f\in C(A_{q,a}^*)$ then $I_{q,0^+}^{\alpha}f\in C(A_{q,a}^*)$ and
 \be\label{ineq1}   \norm{I_{q,0^+}^{\alpha}f}\leq \frac{a^{\alpha}}{\Gamma_q(\alpha+1)}\norm{f}.\ee
\item If $f\in L_q^1(X)$ then $I_{q,0^+}^{\alpha}f\in L_q^1(X)$ and
\be\label{ineq3}\norm{I_{q,0^+}^{\alpha}f}_1\leq M_{\alpha,1}\norm{f}_1,\qquad M_{\alpha,1}:=\dfrac{a^{\alpha}(1-q)^{\alpha}}{(1-q^{\alpha})(q;q)_{\infty}}.
\ee
\item If $f\in L_q^2(X)$ then $I_{q,0^+}^{\alpha}f\in L_q^2(X)$ and
\be\label{ineq2}
\norm{I_{q,0^+}^{\alpha}f}_2\leq\,M_{\alpha,2}\norm{f}_2,\ee
where
\[
M_{\alpha,2}:=\frac{a^{\alpha}}{\Gamma_q(\alpha)}\sqrt{\frac{(1-q)}{(1-q^{2\alpha})}}\left(\int_{0}^{1}(q\xi;q)^2_{\alpha-1}\,d_q\xi\right)^{1/2}.
\]
\item If $\alpha>\frac{1}{2}$ and  $f\in L_q^2(X)$ then $I_{q,0^+}^{\alpha}f\in C(X)$ and
\be\label{ineq0}
\norm{I_{q,0^+}^{\alpha}f}\leq \widetilde{M}_{\alpha}\norm{f},\;\widetilde{M}_{\alpha}:= \frac{a^{\alpha-\frac{1}{2}}}{\Gamma_q(\alpha)}\left(\int_{0}^{1}(q\xi;q)^2_{\alpha-1}\,d_q\xi\right)^{1/2}.
\ee
\item Since $\norm{f}_2\leq \sqrt{a}\norm{f}$,  we conclude that if $f\in C(X)$ then $I_{q,0^+}^{\alpha}f\in L_q^2(X)$ and
\be\label{ineq3-2}
\norm{I_{q,0^+}^{\alpha}f}_2\leq K_{\alpha}\norm{f},\quad K_{\alpha}:=\sqrt{a}M_{\alpha,2}.
\ee
\item If $f\in C(A_{q,a}^*)$ then $I_{q,a^-}^{\alpha}f\in C(A_{q,a}^*)$ and
\[ \norm{I_{q,a^-}^{\alpha}f}\leq c_{\alpha,0} \norm{f},\quad c_{\alpha,0}:=\frac{a^{\alpha}(1-q)^{\alpha}}{(1-q^{\alpha})(q;q)_{\infty}}.\]
\item If $f\in L_q^1(X)$ then $I_{q,a^-}^{\alpha}f\in L_q^1(X)$ and
\[\norm{I_{q,a^-}^{\alpha}f}_1\leq \left\{\begin{array}{cc}\dfrac{(1-q)^{\alpha}a^{\alpha}}{(1-q^{\alpha})(q;q)_{\infty}}\norm{f}_1,& \mbox{if}\,\alpha<1,\\&\\
\dfrac{(1-q)^{\alpha-1}a^{\alpha-1}}{(q;q)_{\infty}}\norm{f}_1,& \mbox{if}\,\alpha \geq 1.\end{array}\right.\]
\item If $\alpha \neq \frac{1}{2}$  and $f\in L_q^2(X)$ then $I_{q,a^-}^{\alpha}f\in L_q^1(X)$ and
\[\norm{I_{q,a^-}^{\alpha}f}_2\leq \left\{\begin{array}{cc}\dfrac{(1-q)^{\alpha-\frac{1}{2}}a^{\alpha}}{\sqrt{1-q^{2\alpha-1}}(q;q)_{\infty}}\norm{f}_2,& \mbox{if}\,\alpha<\frac{1}{2},\\&\\
\dfrac{(1-q)^{\alpha}a^{\alpha}}{(q;q)_{\infty}\sqrt{(1-q^{2\alpha-1})(1-q^{2\alpha})}}\norm{f}_2,& \mbox{if}\,\alpha > \frac{1}{2}.\end{array}\right.\]
\end{enumerate}
The following lemmas are introduced and proved in~\cite{Zeinab_FSL1}
\begin{lem}\label{Lem:1}Let $\alpha>0$. If
\begin{enumerate}
\item[(a)]   $f\in L_q^1(X)$ and $g$ is a bounded function on $A_{q,a}$,

or
\item[(b)] $\alpha\neq \frac{1}{2}$ and $f,\,g$ are $L_q^2(X)$ functions
\end{enumerate}
 then
\be\label{I2}
\int_{0}^{a}g(x)I_{q,0^{+}}^{\alpha}f(x)\,d_qx
=\int_{0}^{a}f(x)I_{q,{a^-}}^{\alpha}g(x)\,d_qx.\ee

\end{lem}
\begin{lem} Let $\alpha\in(0,1)$.
\begin{itemize}
 \item[(a)] If $g\in L_q^1(A_{q,a}^*)$ such that $I_{q}^{1-\alpha}g\in\mathcal{A}C_q(A_{q,a}^*)$, and  $D_q^i f\in C(A_{q,a}^*)$ ($i=0,1$) then
\be\label{CP}
\int_{0}^{a}f(x)D_{q,0^+}^{\alpha}g(x)\,d_qx=-f(\frac{x}{q})I_{q,0^+}^{1-\alpha}g(x)\Big|_{x=0}^a+\int_{0}^{a}g(x) \CM^{\alpha}f(x)\,d_qx.
\ee
\item[(b)] If $f\in\mathcal{A}C_q(A_{q,a}^*)$, and $g$ is a bounded function on $A_{q,a}^*$ such that $D_{q,a^-}^{\alpha}g\in L_q^1(A_{q,a}^*)$ then
\be\label{CP3}
\int_{0}^{a}g(x){}^cD_{q,0^+}^{\alpha}f(x)\,d_qx=\left(I_{q,a^-}^{1-\alpha}g\right)(\frac{x}{q})f(x)\Big|_{x=0}^a+
\int_{0}^{a}f(x) D_{q,a^-}^\alpha g(x)\,d_qx.
\ee
\end{itemize}
\end{lem}

\section{{\bf Basic Fourier series on $q$-Linear grid and some properties}}
The purpose of this section is to reformulate Cardoso's results of Fourier series expansions for functions defined on  the
$q$-linear grid $\mathcal{A}_q:=\set{q^{n},\;n\in\mathbb{N}_0}$ to functions defined on $q$-linear grids
 $\mathcal{A}_{q,a}:=:=\set{\pm a q^{n},\;n\in\mathbb{N}_0}$, $a>0$.
 Cardoso in \cite{Cardoso011} defined the space of all $q$-linear H\"{o}lder functions on the $q$-linear grid $\mathcal{A}_q$.
We generalize his definition for functions defined on a $q$-linear grid of the form  $\mathcal{A}_{q,a}$, $a>0$.

\begin{defn}
A function $f$ defined on $\mathcal{A}_{q,a}$, $a>0$, is called a  $q$-linear H\"{o}lder of order $\lambda$ if there exists a constant $M>0$ such that
\[
\left|f(\pm aq^{n-1})-f(\pm aq^n)\right|\leq Mq^{n\lambda}, \mbox{for all}\;n\in\mathbb{N}.
\]
\end{defn}

\begin{defn}
The $q$-trigonometric functions $S_q(z)$ and $C_q(z)$ are defined for $z\in\mathbb{C}$ by, see~\cite{BC,Cardoso011}
\begin{eqnarray*}
S_q(z)&=&\sum_{n=0}^{\infty}(-1)^n \dfrac{q^{n(n+\frac{1}{2})}z^{2n+1}}{(q;q)_{2n+1}}=\frac{z}{1-q}{}_1\phi_1\left(0;q^3;q^,q^{3/2}z^2\right),\\
C_q(z)&=&\sum_{n=0}^{\infty}(-1)^n \dfrac{q^{n(n-\frac{1}{2})}z^{2n}}{(q;q)_{2n}}={}_1\phi_1\left(0;q;q^,q^{1/2}z^2\right).
\end{eqnarray*}
\end{defn}
One can verify that
\begin{eqnarray*}D_{q,z} S_q(wz)&=&\frac{w}{1-q}C_q(\sqrt{q}z),\\
 D_{q,z} C_q(wz)&=&-\frac{w}{1-q}S_q(\sqrt{q}z),
 \end{eqnarray*}
where $z\in\mathbb{C}$  and $w\in\mathbb{C}$ is a fixed parameter.
A modification of the orthogonality relation given in \cite[Theorem 4.1]{BC} is

\begin{thm} Let $w$ and $w'$ be roots of $S_q(z)$, and
$\mu(w):=(1-q)C_q(q^{1/2}w)S'_q(w)$. Then
\begin{eqnarray*}
\int_{-a}^{a}C_q(\dfrac{q^{\frac{1}{2}}wx}{a})C_q(\dfrac{q^{\frac{1}{2}}w'x}{a})\,d_qx&=&\left\{\begin{array}{cc}0,&\mbox{if}\,w\neq w',\\
2a,&\mbox{if}\, w=w'=0,\\
a\mu(w),&\mbox{if}\,w=w'\neq 0,\end{array}\right.
\\
\int_{-a}^{a}S_q(\dfrac{qwx}{a})S_q(\dfrac{qw'x}{a})\,d_qx&=&\left\{\begin{array}{cc}0,&\mbox{if}\,w\neq w',\\
aq^{-1/2}\mu(w),&\mbox{if}\, w= w'.
\end{array}\right.
\end{eqnarray*}

\end{thm}
Cardoso introduced a sufficient condition for the uniform convergence of the basic Fourier series
\[
S_q(f):=\frac{a_0}{2}+\sum_{k=1}^{\infty}a_k \text{C}_q(q^{1/2}w_k x)+b_k\text{S}_q(qw_kx),
\]
where $a_0=\int_{-1}^{1}f(t)\,d_qt$ and for $k=1,2,\ldots$,
\[
a_k=\frac{1}{\mu_k}\int_{-1}^{1}f(t)C_q(q^{1/2}w_k t)\,d_qt,\quad b_k=\frac{1}{\mu_k}\int_{-1}^{1}f(t)S_q(q w_k t)\,d_qt,\]
\[\mu_k=(1-q)C_q(q^{1/2}w_k)S_q'(w_k)\]
on the $q$-linear grid $\mathcal{A}_{q}$, where $\set{w_k\;:\;k\in\mathbb{N}}$ is the set of positive zeros of $S_q(z)$.
Cardoso proved that $\mu_k=O(q^{-2k^2})$ as $k\to\infty$ for any $q\in(0,1)$.
 In the following we give a modified version of Cardoso's result for any function defined on the $q$-linear grid $\mathcal{A}_{q,a}$, $a>0$.
\begin{thm}\label{MR}
If $f\in C(\mathcal{A}_{q,a}^*)$ is a $q$-linear H\"{o}lder function of order $\lambda>\frac{1}{2}$, then the $q$-Fourier series
\be\label{qfs}
S_q(f):=\frac{a_0}{2}+\sum_{k=1}^{\infty}a_k \text{C}_q(q^{1/2}\frac{w_k x}{a})+b_k\text{S}_q(q\frac{w_k\,x}{a}),
\ee
where
$a_0=\frac{1}{a}\int_{-a}^{a}f(t)\,d_qt$ and for $k=1,2,\ldots$,
\[
a_k(f)=\frac{1}{a\mu_k}\int_{-a}^{a}f(t)C_q(q^{1/2}\frac{w_k t}{a})\,d_qt,\;\;
b_k(f)=\frac{\sqrt{q}}{a\mu_k}\int_{-a}^{a}f(t)S_q(q \frac{w_k t}{a})\,d_qt,
\]
converges uniformly to the function $f$ on the $q$-linear grid $\mathcal{A}_{q,a}$.
\end{thm}
\begin{proof}
The proof is a modification of  the proof of \cite[Theorem 4.1]{Cardoso011} and is omitted.
\end{proof}
\begin{rem}
We replaced the condition $f(0^+)=f(0^-)$  where
\[f(0^+):=\lim_{x\to 0^+} f(x),\quad f(0^-):=\lim_{x\to 0^-} f(x),\]
 in  \cite[Theorem 4.1]{Cardoso011} by the weakest condition that $f$ is $q$-regular at zero.
Because he needs this condition only to guarantee that  $\lim_{n\to\infty} f(q^{n-1/2})=\lim _{n\to\infty} f(-q^{n-1/2})$ and this holds if $f$is $q$-regular at zero. See~\cite[(1.22)]{AMbook}  for a function which is $q$-regular at zero but not continuous at zero.
\end{rem}

A modified version of ~\cite[Theorem 3.5]{Cardoso011} is
\begin{thm}\label{Theorem B}
If there exists  $c>1$  such that
\[\int_{-a}^{a}f(t)\text{C}_q(\sqrt{q}\frac{w_k t}{a})=O(q^{ck})\;\mbox{and}\;\int_{-a}^{a}f(t)\text{S}_q(q\frac{w_k t}{a})=O(q^{ck})\quad \mbox{as}\;k\to\infty,\]
then the $q$-Fourier series \eqref{qfs} converges uniformly on $\mathcal{A}_{q,a}$.
\end{thm}
A modified version of ~\cite[Corollary 4.3]{Cardoso011} is
\begin{cor}
If $f$ is continuous and piecewise smooth  on a neighborhood of the origin,
then the corresponding $q$-Fourier series $S_q(f)$ converges uniformly to $f$ on the set of points $\mathcal{A}_{q,a}$.
\end{cor}

\begin{thm}
If $f\in C(\mathcal{A}_{q,a}^*)$ is a $q$-linear H\"{o}lder  odd function of order $\lambda>\frac{1}{2}$ and  satisfying $f(0)=f(a)=0$, then the $q$-Fourier series
\[
S_q(f):=\sum_{k=1}^{\infty}c_k\text{S}_q(\frac{w_k\,x}{a}),
\]where
\[
c_k(f)=c_k=\frac{2}{a\sqrt{q}\mu_k}\int_{0}^{a}f(t)S_q( \frac{w_k t}{a})\,d_qt,
\]
converges uniformly to the function $f$ on the $q$-linear grid $\mathcal{A}_{q,a}$.
\end{thm}
\begin{proof}
The proof follows from \eqref{MR} by considering the function
$g(x):=f(qx)$, $x\in \mathcal{A}$. Since, it is odd, we have $a_k=0$ for $k=0,1,\ldots$, and
\[b_k(f)=\sqrt{q}{\mu_k}\int_{-a}^{a}g(t)S_q( \frac{qw_k t}{a})\,d_qt,\]
making the substitution $u=qt$ and using that $g$ is an odd function, we obtain the required result.
\end{proof}
\begin{defn}
Let $(f_n)_n$ be a sequence of functions in $C(\mathcal{A}_{q,a}^*)$. We say that $f_n$ converges to a function $f$ in $q$-mean
if \[\lim_{n\to\infty} \sqrt{\int_{-a}^{a}|f_n(x)-f(x)|^2\,d_qx}=0.\]
\end{defn}

\begin{prop}\label{prop:FT}
If $g\in C(\mathcal{A}_{q,a}^*)$ is an odd function satisfying  $D_q^kg$ $(k=0,1,2)$ is continuous and piecewise smooth function  in a neighborhood of zero, and satisfying the boundary condition
\begin{align}
g(0)=g(a)=0
\end{align}
then  $g$ can be approximated in the $q-$mean by a linear combination
\[g_n(x)=\sum_{r=}^{n}c_r^{(n)}\text{S}_q(\frac{w_k x}{a})\]
where at the same time $D_q^k g_n$ $($k=1,2$)$ converges in $q$-mean to the $D_q^k g$. Moreover, the coefficients $c_r^{(n)}$ need not depend on $n$ and can be
 written simply as $c_r$.
\end{prop}

\begin{proof}
We consider the $q$-sine Fourier transform of $D_q^2g$. Hence
\be\label{feq1} D_q^2g(x)=\sum_{k=1}^{\infty}b_k S_q(\frac{q w_k x}{a})=\lim_{n\to\infty}\gamma_n(x),\quad x\in A_{q,a},\ee
where
\[\gamma_n(x)=\sum_{k=1}^{n}b_k S_q({qw_kx}{a}),\quad b_k=\frac{\sqrt{q}}{a\mu_k}\int_{0}^{a}D_q^2g(x)\text{S}_q(\frac{qw_k x}{a})\,d_qx.\]
Consequently,
\[\lim_{n\to\infty}\int_{0}^{a}\left|D_q^2g(x)-\gamma_n(x)\right|^2\,d_qx=0.\]
Hence
\[D_qg(x)-D_qg(0)=\int_{0}^{x}D_q^2g(x)\,d_qx=\frac{a(1-q)}{\sqrt{q}}\sum_{k=1}^{\infty}\frac{b_k}{w_k} \left(-\text{C}_q(\frac{q^{1/2}w_k x}{a})+1\right).\]
Applying the $q$-integration by parts rule~\eqref{qIR} gives
\[a_k(D_q g)=-\frac{a(1-q)}{\sqrt{q}w_k}b_k(D_q^2 g).\]
\[\mbox{I.e.}\qquad  D_qg(x)-D_qg(0)=\sum_{k=1}^{\infty}a_k(D_q g) \left(\text{C}_q(\frac{q^{1/2}w_k x}{a})-1\right).\]
Hence
\be\label{feq2} D_q\,g(x)=\sum_{k=1}^{\infty}a_k(D_q g)\text{C}_q(\frac{q^{1/2}w_k x}{a}), \;x\in A_{q,a}^*\ee
Note that $a_0(D_qg)=0$ because $g(0)=g(a)=0$. Again by $q$-integrating the two sides of \eqref{feq2}, we obtain
\be \label{feq3}g(x)=\sum_{k=1}^{\infty}a_k(D_qg)\frac{a(1-q)}{w_k}S_q(\frac{w_k x}{a}),\quad x\in A_{q,a}^*\ee
One can verify that \[b_k(g)=\frac{a(1-q)}{w_k}a_k(D_qg).\] Hence the right hand sides of \eqref{feq2} and \eqref{feq3} are the $q$-Fourier series of
 $D_qg$ and $g$, respectively. Hence the convergence is uniform in $C(\mathcal{A}_{q,a}^*)$ and $L_{q^2}(\mathcal{A}_{q,a}^*)$ norms.

\end{proof}

\section{{$q$-Fractional Variational Problems}}\label{Sec:CV1}

The calculus of variations is as old as the calculus itself, and has many applications in physics and mechanics.  As the calculus has two forms, the continuous  calculus with the power concept of limits, and the discrete calculus which also called the calculus of finite difference, the calculus of variations has also both of
 the discrete and continuous forms. For a brief history of the continuous calculus of variations, see~\cite{Brunt}. The discrete calculus of variations starts in 1948 by Fort
 in his book~\cite{Fort} where he devoted a chapter to the finite analogue of the calculus of variations, and he introduced a necessary condition analogue to the Euler
 equation and also a sufficient condition.
 The paper of Cadzow~\cite[1969]{Cadzow} was the first paper published in this field, then
  Logan developed the theory in his Ph.D. thesis~\cite[1970]{Loganthesis} and in a series of papers~\cite{Logan72,Logan73-1,Logan73-2,Logan74}.
   See also the Ph.D. thesis of Harmsen~\cite{DCV} for a brief history for  the discrete variational calculus, and for the developments in the theory,
   see~\cite{Harmsen95,Harmsen96-1,Harmsen96-2,CC,Cheung,CDZ2013,Cristina}. In 2004, a $q$-version of the discrete variational calculus is firstly introduced by
   Bangerezako in~\cite{Bangerezako1} for a functions defined in the form
   \[   J(y(x))=\int_{q^{\alpha}}^{q^\beta}xF(x,y(x),D_qy(x),\ldots, D_q^ky(x))\,d_qx,
   \]
   where $q^{\alpha}$ and $q^{\beta}$ are in the uniform lattice $A_{q,a}^*$ for some $a>0$ such that $\alpha>\beta$, provided that the boundary conditions
   \[D_q^{j}y(q^{\alpha})=D_q^{j}y(q^{\beta+1})=c_j\quad(j=0,1,\ldots,k-1).\]
   He introduced a $q$-analogue of the Euler-Lagrange equation and applied it to solve certain isoperimetric problem. Then, in 2005,
    Bangerezako~\cite{Bangerezako2} introduced certain $q$-variational problems on a nonuniform lattice. In 2010, Malinowska, and  Torres introduced the Hahn quantum
    variational calculus. They derived the Euler-Lagrange equation associated with variational problem
    \[
    J(y)=\int_{a}^{b}F\left(t,y(qt+w),D_{q,w}y(t)\right)\,d_{q,w}t,
    \]
    under the boundary condition $y(a)=\alpha$, $y(b)=\beta$ where $\alpha$ and $\beta $ are constants and $D_{q,w}$ is the Hahn difference operator defined by
    \[
    D_{q,w}f(t)=\left\{\begin{array}{cc}\dfrac{f(qt+w)-f(t)}{(qt+w)-t},& \mbox{if}\; t\neq \frac{w}{1-q},\\
    f'(0),&\mbox{ if}\; t=\frac{w}{1-q}.\end{array}\right.
    \]
   Problems of the classical calculus of variations with integrand depending on fractional derivatives  instead of ordinary derivatives are first introduced by
    Agrawal~\cite{Agarwal2002} in 2002. Then, he extends his result for variational problems include Riez fractional derivatives in~\cite{Agrawal2007}. Numerous works     have been dedicated  to the subject since Agarwal work. See for example~
    \cite{Hamilton1,Hamiliton2,KOM,Almedia-Malinowska-Torres,Almedia-Torres1,Almedia-Torres2,Almedia-Torres3,El-Nabulsi}.

    In this Section, we shall derive
    Euler--Lagrange equation for a $q$-variational problem when the integrand include  a left-sided $q$-Caputo fractional derivative and we also solve a related isoperimetric problem. From now on,  we fix $\alpha\in(0,1)$, and define  a subspace  of $C(A_{q,a}^*)$ by
 \[{}_0E_{a}^{\alpha}=\set{y\in \mathcal{A}C(A_{q,a}^*)\;:\; \CZ^{\alpha}y\in\,C(A_{q,a}^*)  },\]
and the space of variations $\V$ for the  Caputo $q$-derivative  by
\[\V=\set{h\in {}_0E_{a}^{\alpha}\;:\; h(0)=h(a)=0}.\]
For a function $f(x_1,x_2,\ldots,x_n)$ ($n\in\mathbb{N}$) by $\partial_i f$ we mean the partial derivative of  $f$ with respect to the $i$th variable, $i=1,2,\ldots,n$.
In the sequel, we shall need the following definition from~\cite{HVP}.
\begin{defn} Let $A\subseteq \mathbb{R}$ and $g:A\times ]-\overline{\theta},\overline{\theta}[\to\mathbb{R}$. We say that $g(t,\cdot)$ is continuous at  $\theta_0$
uniformly in $t$, if and only if $\forall \epsilon>0\; \exists \delta>0$ such that
\[|\theta-\theta_0|<\delta\To |g(t,\theta)-g(t,\theta_0)|<\epsilon\;\mbox{ for all}\; t\in A.\] Furthermore, we say that $g(t,\cdot)$ is differentiable
at $\theta_0$ uniformly in $t$
if and only if $\forall \epsilon>0$ $\exists \delta>0$ such that
\[|\theta-\theta_0|<\delta\To \left|\frac{g(t,\theta)-g(t,\theta_0)}{\theta-\theta_0}-\delta_2g(t,\theta_0)\right|<\epsilon\;\mbox{for all}\; t\in A.\]

\end{defn}
We now present first order necessary conditions of optimality for functionals, defined on $\E$, of the type
\be\label{J:fn}
J(y)=\int_{0}^{a}F(x,y,\CZ^{\alpha}y)\,d_qx,\quad 0<\alpha<1,
\ee
where  $F:A_{q,a}^*\times \mathbb{R}\times \mathbb{R}\to\mathbb{R}$ is a given function.
We assume that
\begin{enumerate}
\item The functions
$(u,v)\to F(t,u,v)$ and $(u,v)\to  \partial_i F(t,u,v)$ $(i=2,3)$ are continuous functions uniformly on $A_{q,a}$.
\item $F\left(\cdot,y(\cdot),\CZ^{\alpha}(\cdot)\right)$, $\delta_i F\left(\cdot,y(\cdot),\CZ^{\alpha}(\cdot)\right)$ $(i=2,3)$ are $q$-regular at zero.
\item  $\delta_3 F$ has a right Riemann-–Liouville fractional $q$-derivative of order
$\alpha$ which is $q$-regular at zero.
\end{enumerate}
\begin{defn}
Let $y_0\in \E$. Then $J$ has a local maximum at $y_0$ if
\[ \exists \delta>0\;\mbox{ such that}\;J(y)\leq J(y_0)\;\mbox{ for all}\; y\in \E\;\mbox{ with}\; \norm{y-y_0}<\delta,\] and $J$ has a local minimum at $y_0$ if
\[ \exists \delta>0\;\mbox{ such that}\; J(y)\geq J(y_0)\;\mbox{ for all}\;y\in S\;\mbox{ with}\; \norm{y-y_0}<\delta.\]
$J$ is said the have a local extremum at $y_0$ if it has either a local maximum or local minimum.
\end{defn}

\begin{lem}\label{lem4}
Let  $\gamma\in L_{q}^2(A_{q,a}^*)$.
 \begin{enumerate}
 \item[(i)] If \be\label{INT} \int_{0}^{a}\gamma (x)h(x)\,d_qx=0\ee
for every $h\in L_q^2(A_{q,a})$ then
\be \label{C2}\gamma(x)\equiv 0\;\mbox{on}\;A_{q,a}.\ee
\item[(ii)]If \eqref{INT} holds only for all functions $h\in L_q^2(A_{q,a}^*)$ satisfying  $h(a)=0$
then
\be \label{C2-1}\gamma(x)\equiv 0\;\mbox{on}\;A_{q,qa}.\ee
Moreover, in the two cases,  if $\gamma$ is $q$-regular at zero, then  $\gamma(0)=0$.
\end{enumerate}
\end{lem}

\begin{proof}
To prove (i), we fix $k\in\mathbb{N}_0$ and set  $h_k(x)=\left\{\begin{array}{cc}1,& x=aq^k\\ 0,&\mbox{otherwise}\end{array}\right.$. Then $h_k\in L_q^2(0,a)$. Substituting in \eqref{INT} yields
\[aq^k(1-q)\gamma(aq^k)=0.\quad \forall k\in\mathbb{N}_0. \]
Thus, $\gamma(aq^k)=0$ for all $k\in\mathbb{N}_0$. Clearly if $\gamma$ is $q$-regular at zero, then
\[\gamma(0):=\lim_{k\to\infty}\gamma(aq^k)=0.\]
The proof of (ii) is similar and is omitted.
\end{proof}

\begin{lem}\label{lem0}
If $\alpha\in C(A_{q,a}^*)$ and
\[
\int_{0}^{a}\alpha(x)D_qh(x)\,d_qx=0,
\]
for any function $h$ satisfying
\begin{enumerate}
\item $h$ and $D_q h$ are $q$-regular at zero,
\item $h(0)=h(a)=0$,
\end{enumerate}
then $\alpha(x)=c$ for all $x\in A_{q,a}^*$ where $c$ is a constant.
\end{lem}
\begin{proof}
Let $c$ be the constant defined by the relation  $c=\frac{1}{a}\int_{0}^{a}\alpha(x)\,d_qx$.
Let \[h(x):=\int_{0}^{x}\left[\alpha(\xi)-c\right]\,d_q\xi,\quad x\in A_{q,a}^*.\]
So, $h$ and $D_q h$ are $q$-regular at zero functions such that $h(0)=h(a)=0$.
Since
\[\int_{0}^{a}\left[\alpha(x)-c\right]D_qh(x)\,d_qx=\int_{0}^{a}\alpha(x)D_qh(x)\,d_qx+\left[\alpha(x)-c\right] h(x)\Big|_{x=0}^{a}=0,\]
on the other hand,
\[\int_{0}^{a}\alpha(x)D_qh(x)\,d_qx=\int_{0}^{a}\left[\alpha(x)-c\right]^2\,d_qx=0.\]
Therefore, $\alpha(x)=c$ for all $x\in A_{q,a}$.
But $\alpha$ is $q$-regular at zero, hence $\alpha(0)=0$.
This yields the required result.

\end{proof}

\begin{thm}
Let $y\in \V$ be  a local extremum of $J$. Then, $y$ satisfies the Euler-Lagrange equation
\be \label{ELE0}\partial_2 F(x) +D_{q,a^-}^{\alpha}\,\partial_3 F(x)=0,\quad \forall\,x\in A_{q,qa}^*.\ee
\end{thm}

\begin{proof}
 Let $y$ be a local extremum of $J$ and let  $\eta$ be arbitrary but fixed  variation function  of $y$.
 Define \[\Phi(\epsilon)=J(y+\epsilon\eta).\]
Since $y$ is a local extremum for $J$, and $J(y)=\Phi(0)$, it follows that $0$ is a local extremum for $\phi$. Hence $\phi'(0)=0$.
But \[0=\phi'(0)=\lim_{\epsilon\to 0}\frac{d}{d\epsilon}\phi(y+\epsilon \eta)=\int_{0}^{a}\left(\partial_2 F \eta+\partial_3 F\,\CZ^{\alpha}\eta\right)\,d_qx.\]
 Using \eqref{CP}, we obtain
\[0=\int_{0}^{a}\left(\partial_2 F +\CM^{\alpha}\partial_3 F\right)\eta\,d_qx+I_{q,a^-}^{1-\alpha}\partial_3 F(x) \eta(x)\Big|_{x=0}^{a}.\]
Since $\eta$ is a variation function, then $\eta(0)=\eta(a)=0$ and
we have
\[\int_{0}^{a}\left(\partial_2 F +D_{q,a^-}^{\alpha}\,\partial_3 F\right)\eta\,d_qx=0\] for any $\eta\in S$. Consequently, From Lemma~\ref{lem4}, we obtain
\eqref{ELE0} and completes the proof.
 \end{proof}

\subsection{A $q$-Fractional Isoperimetric Problem}\label{Sec:CV2}
In the following,  we shall  solve the $q$-fractional isoperimetric  problem:
Given a functional $J$ as in \eqref{J:fn}, which function minimize (or maximize) $J$, when subject to the boundary conditions
\be\label{EQ:2}
y(0)=y_0,\quad y(a)=y_a
\ee
and the $q$-integral constraint
\be\label{EQ:3}
I(y)=\int_{0}^{a}G(x,y, \CZ^\alpha y)\,d_qx=l,
\ee
where $l$ is a fixed real number. Here, similarly as before,
\begin{enumerate}
\item The functions
$(u,v)\to G(t,u,v)$ and $(u,v)\to  \partial_i G(t,u,v)$ $(i=2,3)$ are continuous functions uniformly on $A_{q,a}$.
\item $G\left(\cdot,y(\cdot),\CZ^{\alpha}(\cdot)\right)$, $\delta_i G\left(\cdot,y(\cdot),\CZ^{\alpha}(\cdot)\right)$ $(i=2,3)$ are $q$-regular at zero.
\item  $\delta_3 G$ has a right Riemann-–Liouville fractional $q$-derivative of order
$\alpha$ which is $q$-regular at zero.
\end{enumerate}
A function $y\in E$ that satisfies \eqref{EQ:2} and \eqref{EQ:3} is called admissible.

\begin{defn}
An admissible function $y$ is an extremal for $I$ in \eqref{EQ:3}
 if it satisfies the equation
 \be \label{ELE}\partial_2 G(x) +D_{q,a^-}^{\alpha}\,\partial_3 G(x)=0,\quad \forall\,x\in A_{q,qa}^*.\ee
\end{defn}
\begin{thm}\label{thm:qIP}
Let $y$ be a local extremum  for $J$ given by \eqref{J:fn}, subject to the conditions \eqref{EQ:2} and \eqref{EQ:3}. If $y$ is not an extremal of the function $I$, then there exits a constant $\lambda$ such that $y$ satisfies
\be \label{ELE2}
\partial_2 H(x) +D_{q,a^-}^{\alpha}\,\partial_3 H(x)=0,\quad \forall\,x\in A_{q,qa}^*,\ee
where $H:=F-\lambda\,G$.
\end{thm}
\begin{proof}
Let $\eta_1$, $\eta_2\in \V$ be two functions, and let $\epsilon_1$ and $\epsilon_2$ be two real numbers, and consider the new function of two parameters.
\be\label{EQ:4}
\breve{y}=y+\epsilon_1\eta_1+\epsilon_2\eta_2.
\ee
The reason why we consider two parameters is because we can choose one of them as a function of the other in order to $\breve{y}$ satisfies  the $q$-integral constraint \eqref{EQ:3}. Let
\[
\breve{I}(\epsilon_1,\epsilon_2)=\int_{0}^{a}G(x, \breve{y},\CZ^{\alpha}\breve{y})\,d_qx-l.
\]
It follows by the $q$-integration by parts rule~\eqref{qIR} that
\[
\frac{\partial \breve{I}}{\partial \epsilon_2}\Big|_{(0,0)}=\int_{0}^{a}\left(\partial_2 G(x) +D_{q,a^-}^{\alpha}\,\partial_3 G(x)\right)\eta_2\,d_qx.
\]
Since $y$ is not an extremal of $I$, then there exists a function $\eta_2$ satisfying the condition
$\frac{\partial \breve{I}}{\partial \epsilon_2}|_{(0,0)}\neq 0$.  Hence, from the fact that $\breve{I}(0,0)=0$ and the {\it Implicit Function Theorem}, there
 exists a $C^1$ function $\epsilon_2(\cdot)$, defined in some neighborhood of zero, such that
\[ \breve{I}(\epsilon_1, \epsilon_2(\epsilon_1))=0.\]
Therefore, there exists a family of variations of type \eqref{EQ:4} that satisfy the $q$-integral constraint. To prove the theorem, we define a new function $\breve{J}(\epsilon_1,\epsilon_2)=J(\breve{y})$. Since $(0,0)$ is a local extremum of $\breve{J}$
subject to the constraint $\breve{I}(0,0)=0$, and
$\nabla \breve{I}(0,0)\neq (0,0)$, by the Lagrange Multiplier rule, see~\cite{Lang}, there exists a constant $\lambda$ for which the following holds:
\[
\nabla \breve{J}(0,0)\neq (0,0)-\lambda  \breve{I}(0,0)= (0,0).
\]
Simple calculations shows that
\[
\frac{\partial \breve{J}}{\partial \epsilon_1}\Big|_{(0,0)}=\int_{0}^{a}\left(\partial_2 F(x) +D_{q,a^-}^{\alpha}\,\partial_3 F(x)\right)\eta_1\,d_qx,
\]
and
\[
\frac{\partial \breve{I}}{\partial \epsilon_1}\Big|_{(0,0)}=\int_{0}^{a}\left(\partial_2 G(x) +D_{q,a^-}^{\alpha}\,\partial_3 G(x)\right)\eta_1\,d_qx.
\]
Consequently,
\[\int_{0}^{a}\left[\partial_2 F(x) +D_{q,a^-}{\alpha}\,\partial_3 F(x)-\lambda \left(\partial_2 G(x) +D_{q,a^-}{\alpha}\,\partial_3 G(x)\right)\right]\eta_1\,d_qx.\]
Since $\eta_1$ is arbitrary, then from Lemma~\ref{lem4}, we obtain
\[\partial_2 F(x) +D_{q,a^-}^{\alpha}\,\partial_3 F(x)-\lambda \left(\partial_2 G(x) +D_{q,a^-}^{\alpha}\,\partial_3 G(x)\right)=0,\]
for all $x\in A_{q,qa}^{*}$.
This is equivalent to \eqref{ELE2} and completes the proof.
\end{proof}
The functions
\begin{eqnarray*}e_{\alpha,\beta}(z;q)&:=&\sum_{n=0}^{\infty}\dfrac{z^n}{\Gamma_q(\alpha\,n+1)};\;|z(1-q)^{\alpha}|<1,\\ E_{\alpha,\beta}(z;q)&:=&\sum_{n=0}^{\infty}q^{\frac{\alpha}{2}n(n-1)}\dfrac{z^n}{\Gamma_q(\alpha n+1)};\;z\in\mathbb{C}
 \end{eqnarray*}
are $q$-analogues of the Mittag--Leffler function
\[E_{\alpha,\beta}(z)=\sum_{n=0}^{\infty}\dfrac{z^n}{\Gamma_q(\alpha\,n+1)},\; z\in\mathbb{C},\]
see~\cite{AMbook}.
We have
\be\label{LSD} \CZ^{\alpha}e_{\alpha,1}(z;q):=e_{\alpha,1}(z;q);\quad \CZ^{\alpha}E_{\alpha,1}(z;q)=E_{\alpha,1}(qz;q).\ee

\begin{exa}
Consider the fractional $q$-isoperimetric problem:
\be\label{IP:1}\begin{split}
J(y)&=\int_{0}^{a}\left(\CZ^{\alpha}y(x)\right)^2\,d_qx,\\
I(y)&=\int_{0}^{a}e_{\alpha,1}(x^{\alpha};q)\CZ^{\alpha}y(x)\,d_qx=l,\\
\quad&y(0)=1,\qquad y(a)=e_{\alpha,1}(a^{\alpha};q),
\end{split}\ee
where $0<a(1-q)<1$.
Then \[H=\left(\CZ^{\alpha}y\right)^2-\lambda e_{\alpha,1}(x;q)\CZ^{\alpha}y,\]
and \[\partial_2 H+D_{q,a^-}^{\alpha}\partial_3H=D_{q,a^-}^{\alpha}\left(2\CZ^{\alpha}y(x)-\lambda~e_{\alpha,1}(x;q)\right).\]
 Therefore  a solution of the problem is $\lambda=2$ and $y(x)=e_{\alpha,1}(x^{\alpha};q)$ . Similarly
a solution of the problem
\be\label{IP:2}\begin{split}
J(y)&=\int_{0}^{a}\left(\CZ^{\alpha}y(x)\right)^2\,d_qx,\\
I(y)&=\int_{0}^{a}E_{\alpha,1}((qx)^{\alpha};q)\CZ^{\alpha}y(x)\,d_qx=l,\\
\quad&y(0)=1,\qquad y(a)=E_{\alpha,1}(a^{\alpha};q),
\end{split}\ee
where $a>0$ is $y(x)=E_{\alpha,1}(x^{\alpha};q)$.
 \end{exa}
\section{Existence of Discrete Spectrum for a fractional $q$-Sturm--Liouville problem}

In this section, we  use the  $q$-calculus of variations we developed in Sections \ref{Sec:CV1}  to investigate the existence of solutions of the qFSLP
\be\label{EVP3}
D_{q,a^-}^{\alpha}p(x)\CZ^{\alpha}y(x) +r(x)y(x)=\lambda w_{\alpha}y(x),\quad x\in A_{q,qa}^*,
\ee
under the boundary condition
\be\label{BC6} y(0)=y(a)=0.
\ee
The proof of the main result of this section depends on Arzela-Ascoli Theorem~\cite[P. 156]{Rud64}. The setting of this theorem  is a compact metric space $X$. Let $C(X)$ denote the space of all continuous functions on
$X$ with values in $\mathbb{C}$ or $\mathbb{R}$. $C(X)$ is associated with the metric function
\[d(f, g) = \max\set{|f((x)- g(x)|\; :\; x\in X}.\]

\begin{thm}[Arzela-Ascoli Theorem]
If a sequence $\set{f_n}_n$ in $C(X)$ is bounded
and equicontinuous then it has a uniformly convergent subsequence.
\end{thm}

In our $q$-setting, we take $X=A_{q,a}^*$. Hence   $f\in C(A_{q,a}^*) $ if and only if  $f$ is $q$-regular at zero, i.e.
\[f(0):=\lim_{n\to\infty}f(aq^n).\]
\begin{rem}
A question may be raised why  in \eqref{EVP3} we have only  $x\in A_{q,qa}^*$ instead of $A_{q,a}^*$.
 The reason for that is the  qFSLP~\eqref{EVP3}--\eqref{BC6}  will be solved  by using  the $q$-fractional isoperimetric problem developed in Theorem~\ref{thm:qIP},
  and its  $q$-Euler--Lagrange equation~\eqref{ELE2} holds only for $x\in A_{q,qa}^*$.
  Also,  in order that \eqref{EVP3} holds at $x=a$, we should have $D_{q,a^-}^{\alpha}(p(\cdot)\CZ^{\alpha}y(\cdot))(a)=0$ and
  this holds only if $p(a)\CZ^{\alpha}y(a)=0$ which may not hold.

\end{rem}

\begin{thm}
Let $\frac{1}{2}<\alpha<1$. Assume that the functions $p,\,r,\, w_{\alpha}$  are defined on $A_{q,a}^*$ and  satisfying the conditions
\begin{enumerate}
\item[(i)] $w_{\alpha}$ is a positive continuous function on $[0,a]$ such that $D_q^{k}\frac{1}{w_{\alpha}}$ $(k=0,1,2)$ are bounded functions on $A_{q,a}$,
\item[(ii)] $r$ is a bounded function on $A_{q,a}$,
\item[(iii)]  $p\in C(A_{q,a}^*)$ such that  $\ds\inf_{x\in A_{q,a}}p(x)>0$, and  $\ds\sup_{x\in A_{q,a}}\left|\frac{r(x)}{w_{\alpha}(x)}\right|<\infty$.
\end{enumerate}
 The $q$-fractional Sturm--Liouville problem \eqref{EVP3}--\eqref{BC6} has an infinite number of eigenvalues $\lambda^{(1)}$, $\lambda^{(2)}$, \ldots, and to each eigenvalue $\lambda^{(n)}$ there is
a corresponding eigenfunction $y^{(n)}$ which is unique up to a constant factor. Furthermore, eigenfunctions $y^{(n)}$ form an orthogonal set of solutions in the Hilbert space
$L_q^2(A_{q,a}^*,w_{\alpha})$.
\end{thm}

\begin{proof}
The qFSLP  \eqref{EVP3}--\eqref{BC6} can be recast as a $q$-fractional  variational problem. Let
\be \label{qF}J(y)=\int_{0}^{a}\left[p(x)\left(\CZ^\alpha y\right)^2+r(x)y^2\right]\,d_qx\ee
and consider the problem of finding the extremals of $J$ subject to the boundary condition
\be\label{qF2}y(0)=y(a)=0,\ee
and the  isoperimetric constraint
\be \label{qF1}I(y)=\int_{0}^{a}w_{\alpha}(x)y^2\,d_qx=1.\ee
The $q$-fractional  Euler-Lagrange equation for the functional $I$ is
\[2 w_{\alpha}(x) y(x)=0,\quad\mbox{ for all}\; x\in A_{q,a}\]
which is satisfied only for the trivial solution $y=0$, because $w_{\alpha}$ is positive on $A_{q,a}$. So, no extremals for $I$ can satisfy  the $q$-isoperimetric
condition. If $y$ is an extremal for the $q$-fractional isoperimetric problem, then from Theorem \ref{thm:qIP}, there exists a constant $\lambda$
 such that $y$ satisfies the $q$-fractional  Euler-Lagrange equation \eqref{ELE2} in $A_{q,qa}^*$ but this is equivalent to the qFSLP \eqref{EVP3}
In the following, we shall derive a method for approximating the eigenvalues and the eigenfunctions at the same time similar to the technique in~\cite{GFbook,KOM}. The proof follows in 6 steps.

\vskip .5 cm
\noindent{\bf Step 1.}
First let us point out that functional \eqref{qF} is
bounded from below. Indeed, since $p,\,w_{\alpha}$ are positive on $A_{q,a}$, then
\[\begin{split}
J(y)&=\int_{0}^{a}\left[p(x)\left(\CZ^\alpha y\right)^2+r(x)y^2\right]\\
&\geq \inf_{x\in A_{q,a}}\frac{r(x)}{w_{\alpha}(x)}\int_{0}^{a}w_{\alpha}(x)y^2(x)\,d_qx=\inf\limits_{x\in A_{q,a}}\frac{r(x)}{w_{\alpha}(x)}=:M>-\infty.
\end{split}
\]
According to Ritz method~\cite[P. 201]{GFbook}, we approximate a solution of \eqref{qF}--\eqref{qF2} using the following $q$-trigonometric functions with the coefficients depending on $w_{\alpha}$:
\be\label{sq}
y_m(x)=\frac{1}{\sqrt{w_{\alpha}}}\sum_{k=1}^{m}\frac{\beta_k}{\sqrt{\mu_k}}S_q(\frac{w_kx}{a}).
\ee
Observe that $y_m(0)=y_m(a)=0$. By substituting \eqref{sq} into \eqref{qF} and \eqref{qF1} we obtain
\be\label{C0}
\begin{gathered}
{J_m}(\beta_1,\ldots,\beta_m)={J_m}([\beta])=\sum_{k,j=1}^m\frac{\beta_j\beta_k}{\sqrt{\mu_j}\sqrt{\mu_k}}\times\\
\int_{0}^{a}\left[p(x)\CZ^{\alpha} \frac{S_q(\frac{w_kx}{a})}{\sqrt{w_{\alpha}}}\CZ^{\alpha} \frac{S_q(\frac{w_jx}{a})}{\sqrt{w_{\alpha}}}+\frac{r(x)}{w_{\alpha}(x)}S_q(\frac{w_kx}{a})S_q(\frac{w_jx}{a})\right]\,d_qx
\end{gathered}
\ee
subject to the condition
\be\label{C}{I_m}(\beta_1,\beta_2,\ldots,\beta_m)={I_m}([\beta])=\frac{a\sqrt{q}}{2}\sum_{k=1}^{m}\beta_k^2=1.
\ee
The functions defined in \eqref{C0} and \eqref{C} are functions of the $m$ variables $\beta_1,\,\beta_2,\ldots,\beta_m$. Thus, in terms of the variables $\beta_1,\ldots, \beta_m$, our problem is to minimize ${J_m}(\beta_1,\beta_2,\ldots,\beta_m)$ on the surface $\sigma_m$ of the $m$ dimensional sphere
defined in \eqref{C}. Since $\sigma_m$ is a compact set and ${J_m}(\beta_1,\beta_2,\ldots,\beta_m)$ is continuous on $\sigma_m$, ${J_m}(\beta_1,\beta_2,\ldots,\beta_m)$ has a minimum $\lambda_m^{{1}}$ at some point $(\beta_1^{(1)},\ldots,\beta_m^{(1)} )$ of $\sigma_m$. Let
\[y_m^{(1)}=\frac{1}{\sqrt{w_{\alpha}}}\sum_{k=1}^{m}\frac{\beta_k^{(1)}}{\mu_k}S_q(\frac{w_kx}{a}).\]
If this procedure is carried out for $m=1,2,\ldots$, we obtain a sequence of numbers $\lambda_1^{(1)},\,\lambda_2^{(1)},\,\ldots,$
and a corresponding sequence of functions
\[y_1^{(1)}(x),\;y_2^{(1)}(x),\;y_3^{(1)}(x),\ldots .\]
Noting that $\sigma_m$ is the subset of $\sigma_{m+1}$ obtained by setting $\beta_{m+1}=0$, while
\[J_{m}(\beta_1,\ldots,\beta_m)=J_{m+1}(\beta_1,\ldots,\beta_m,0),\]
consequently,
\be\label{C3}
\lambda_{m+1}^{(1)}\leq \lambda_m^{(1)}.
\ee
Since increasing the domain of definition of a function can only decrease its minimum. It follows from \eqref{C3} and the fact that $J(y)$ is bounded from below that its limit
\[
\lambda^{(1)}=\lim_{m\to\infty}\lambda_m^{(1)}
\]
exists.

\vskip 0.5 cm

\noindent{\bf Step 2.} We shall prove that the sequence $(y_m^{(1)})_{m\in\mathbb{N}}$ contains a uniformly convergent subsequence. From now on, for simplicity, we shall write $y_m$ instead of $y_m^{(1)}$.  Recall that
\[\lambda_m^{(1)}=\int_{0}^{a}\left[p(x)(\CZ^{\alpha} y_m)^2+r(x)y_m^2\right]\,d_qx\]
is convergent, so it must be bounded, i.e., there exists a constant $M_0>0$ such that
\[\int_{0}^{a}\left[p(x)(\CZ^{\alpha} y_m)^2+r(x)y_m^2\right]\,d_qx\leq M_0,\quad m\in\mathbb{N}.\]
Therefore, for all $m\in\mathbb{N}$ it holds the inequality

\begin{eqnarray*}
\int_{0}^{a}p(x)(\CZ^{\alpha} y_m)^2\,d_qx&\leq& M_0+\left|\int_{0}^{a}r(x)y_m^2(x)\,d_qx\right|\\
&\leq& M_0+\sup_{x\in A_{q,a}}\left|\frac{r(x)}{w_{\alpha}(x)}\right|\int_{0}^{a}w_{\alpha}(x)y_m^2(x)\,d_qx\\
&:=&M_0+\sup_{x\in A_{q,a}}\left|\frac{r(x)}{w_{\alpha}(x)}\right|=:M_1.
\end{eqnarray*}
 Moreover, since $\ds \inf_{x\in A_{q,a}}p(x)>0$ we have
\[\left(\inf_{x\in A_{q,a}}p(x)\right)\int_{0}^{a}(\CZ^{\alpha} y_m)^2\,d_qx\leq \int_{0}^{a}p(x)(\CZ^{\alpha} y_m)^2\,d_qx\leq M_1, \]
and hence
\be \label{i0}\int_{0}^{a}(\CZ^{\alpha} y_m)^2\,d_qx \leq \dfrac{M_1}{\inf\limits_{x\in A_{q,a}}p(x)}=:M_2^2.\ee

Since $y_m(0)=0$, then  from  \eqref{ineq0} and \eqref{i0}
\begin{eqnarray*}
\norm{y_m}&=&\norm{I_{q,0^+}^{\alpha}{}^cD_{q,0^+}^{\alpha} y_m}\leq \widetilde{M}_{\alpha}\norm{{}^cD_{q,0^+}^{\alpha} y_m}_2\\
&\leq&\widetilde{M}_{\alpha} {M_2}.
\end{eqnarray*}
for $\alpha>1/2$. Hence, $(y_m)_m$ is uniformly bounded on $A_{q,a}^*$. Now we prove that the sequence $(y_m)_m$ is equicontinuous.
Let $x_1,x_2\in A_{q,a}$. Assume that $x_1<x_2$. Applying the Schwarz's inequality and \eqref{CR4}
\begin{eqnarray*}
\hskip -1 cm&&\Gamma_q(\alpha)\left|y_m(x_2)-y_m(x_1)\right|=\Gamma_q(\alpha)\left|I_{q,0^+}^{\alpha}\CZ^{\alpha} y_m(x_2)-I_{q,0^+}^{\alpha}\CZ^{\alpha} y_m(x_2)\right|\\
&=&\left|x_2^{\alpha-1}\int_{0}^{x_2}(qt/x_2;q)_{\alpha-1}\CZ^{\alpha}y_m(t)\,d_qt-x_1^{\alpha-1}
\int_{0}^{x_1}(qt/x_1;q)_{\alpha-1}\CZ^{\alpha}y_m(t)\,d_qt\right|\\
&\leq&\left|x_2^{\alpha-1}\int_{x_1}^{x_2}(qt/x_2;q)_{\alpha-1}{}^cD_{q,0^+}^{\alpha}y_m(t)\,d_qt\right|\\
&&+\left|\int_{0}^{x_1}\left\{x_2^{\alpha-1}(qt/x_2;q)_{\alpha-1}-x_2^{\alpha-1}(qt/x_1;q)_{\alpha-1}\right\}{}^cD_{q,0^+}^{\alpha}y_m(t)\,d_qt\right|\\
&\leq& M_2 \left(x_2^{2\alpha-2}\int_{x_1}^{x_2}(qt/x_2;q)_{\alpha-1}^2\,d_qt\right)^{1/2} \\
&+&M_2\left(\int_{0}^{x_1}\left(x_2^{\alpha-1}(qt/x_2;q)_{\alpha-1}-x_1^{\alpha-1}(qt/x_1;q)_{\alpha-1}\right)^2\,d_qt\right)^{1/2}.
\end{eqnarray*}
Since $x_1<x_2$, then we have \[x_2^{\alpha-1}({qt}/{x_2};q)_{\alpha-1}\leq x_1^{\alpha-1} ({qt}/{x_1};q)_{\alpha-1}\quad\mbox{ for all}\;t<x_1<x_2.\]
Using the inequality \[t_1\geq t_2 \geq 0\rightarrow (t_1-t_2)^2\leq t_1^2-t_2^2,\]
we obtain
\begin{eqnarray*}
&&\int_{0}^{x_1}\left(x_2^{\alpha-1}(qt/x_2;q)_{\alpha-1}-x_1^{\alpha-1}(qt/x_1;q)_{\alpha-1}\right)^2\,d_qt\\
&\leq& \int_{0}^{x_1}x_1^{2\alpha-2}(qt/x_1;q)_{\alpha-1}^2\,d_qt-\int_{0}^{x_1}x_2^{2\alpha-2}(qt/x_2;q)_{\alpha-1}^2\,d_qt\\
&=&\int_{x_1}^{x_2}x_2^{2\alpha-2}(qt/x_2;q)_{\alpha-1}^2\,d_qt+\int_{0}^{x_1}x_1^{2\alpha-2}(qt/x_1;q)_{\alpha-1}^2\,d_qt\\
&&-\int_{0}^{x_2}x_2^{2\alpha-2}(qt/x_2;q)_{\alpha-1}^2\,d_qt\\
&=&\int_{x_1}^{x_2}x_2^{2\alpha-2}(qt/x_2;q)_{\alpha-1}^2\,d_qt+\left(x_1^{2\alpha-1}-x_2^{2\alpha-1}\right)\int_{0}^{1}(q\xi;q)^2_{\alpha-1}\,d_q\xi\\
&\leq& \int_{x_1}^{x_2}x_2^{2\alpha-2}(qt/x_2;q)_{\alpha-1}^2\,d_qt
\end{eqnarray*}
for $\alpha>\frac{1}{2}$.  Hence, we have
\begin{eqnarray*}\left|y_m(x_2)-y_m(x_1)\right|&\leq& \frac{2M_2}{\Gamma_q(\alpha)}x_2^{\alpha-1}\left(\int_{x_1}^{x_2}(qt/x_2;q)_{\alpha-1}^2\,d_qt\right)^{1/2}\\
&\leq &\frac{2M_2}{\Gamma_q(\alpha)(q^\alpha;q)_{\infty}^2}x_2^{\alpha-1}\sqrt{x_2-x_1}\leq  \frac{2M_2}{\Gamma_q(\alpha)^2(q^\alpha;q)_{\infty}^2}(x_2-x_1)^{\alpha-\frac{1}{2}}.
\end{eqnarray*}
Hence $\set{y_m}$ is equicontinuous. Therefore, from {\it Arzel\`{a}-Ascoli Theorem} for metric spaces, a uniformly convergent subsequence $(y_{m_n})_{n\in\mathbb{N}}$ exists. It means that we can find $y^{(1)}\in C(A_{q,a}^*)$ such that
\be\label{l1} y^{(1)}=\lim_{n\to\infty}y_{m_n}.\ee

\vskip 0.5 cm

\noindent{\bf Step 3}
From the Lagrange multiplier at $[\beta]=(\beta_1^{(1)},\ldots,\beta_m^{(1)})$, we have
\[\frac{\delta}{\delta \beta_j}\left[ J_m([\beta])-\lambda_{m}^{(1)}I_m([\beta])\right]\Big|_{[\beta]=[\beta^{(1)}]},\;j=1,2,\ldots,m.\]
By multiplying each equation by an arbitrary constant $c_j$ and summing from $1$ to $m$, we obtain
\be\label{LM}
0=\sum_{j=1}^{m}c_j\frac{\delta}{\delta \beta_j}\left[ J_m([\beta])-\lambda_{m}^{(1)}I_m([\beta])\right]\Big|_{[\beta]=[\beta^{(1)}]}.
\ee
For $m\in\mathbb{N}$, set
\[h_m(x):=\frac{1}{\sqrt{w_{\alpha}}}g_m(x);\quad g_m(x):=\sum_{j=1}^{m}\frac{c_j}{\sqrt{\mu_j}}S_q(\frac{w_jx}{a}).\]
According to Proposition \ref{prop:FT}, we can choose the coefficients $c_j$ such that there exists a function $g$ satisfying
\[\lim_{m\to \infty}D_q^k g_m= D_q^k g\quad (k=0,1,2)\]
and the convergence is in $L_q^2(A_{q,a}^*)$ norm. Hence
\be\label{l2}
\lim_{m\to\infty}\norm{D_q^k h_m-D_q^k h}_2=0\quad (k=0,1,2).
\ee
We can write \eqref{LM} in the form
\be\label{eq}
0=\int_{0}^{a}\left[p(x)\CZ^{\alpha} y_m \CZ^{\alpha}h_m+\left(r(x)-\lambda_m^1 w_{\alpha}(x)\right)y_m h_m\right]\,d_qx.
\ee
Since $y_m(0)=0$, then from \eqref{CR6}
\[\CZ^{\alpha}y_m= D_{q,0^+}^{\alpha}y_m=D_q I_{q,0^+}^{1-\alpha}y_m.\]
Then replacing  $\CZ^{\alpha}y_m$ by  $D_{q,0^+}^{\alpha}y_m$ in \eqref{eq}  and applying the $q$-integration by parts rule~\eqref{qIR}, we obtain
\begin{eqnarray*}
0= I_m&:=&-\int_{0}^{a}D_q p(x)\CZ^{\alpha} h_m(x)\left(I_{q,0^+}^{1-\alpha}y_m\right)(qx)\,d_qx\\
&&-\int_{0}^{a}p(qx)D_q \CZ^{\alpha}h_m(x)\left(I_{q,0^+}^{1-\alpha}y_m\right)(qx)\,d_qx\\
&&+\int_{0}^{a}\left[r(x)-\lambda_m^{(1)}w_{\alpha}(x)\right]y_m h_m\,d_qx\\
&&+p(x)\CZ^{\alpha}h_m(x)I_{q,0^+}^{1-\alpha}y_m(x)\Big|_{x=0}^{x=a}.
\end{eqnarray*}
In the following we shall prove that
\be\begin{split}\label{4:15}
I&:=\lim_{m\to\infty}I_m=\int_{0}^{a}-D_q p(x)\CZ^{\alpha} h(x)\left(I_{q,0^+}^{1-\alpha}y^{(1)}\right)(qx)\,d_qx\\
&-\int_{0}^{a}p(qx)D_q \CZ^{\alpha}h(x)\left(I_{q,0^+}^{1-\alpha}y^{(1)}\right)(qx)\,d_qx\\
&+p(x)\CZ^{\alpha}h(x)I_{q,0^+}^{1-\alpha}y(x)\Big|_{x=0}^{x=a}+\int_{0}^{a}\left[r(x)-\lambda^{(1)}w_{\alpha}(x)\right]y^{(1)}h\,d_qx.
\end{split}
\ee
Indeed,
\be\label{38}\begin{gathered}
|I_m-I|\leq \\
\int_{0}^{a}\left|D_q p(x)\left[\CZ^{\alpha}h_m(x)\left(I_{q,0^+}^{1-\alpha}y_m\right)(qx)-\CZ^{\alpha} h(x)\left(I_{q,0^+}^{1-\alpha}y^{(1)}\right)(qx)\right]\right|\,d_qx\\
+\int_{0}^{a}\left|p(qx)\left[D_q\CZ^{\alpha}h_m(x)\left(I_{q,0^+}^{1-\alpha}y_m\right)(qx)-D_q\CZ^{\alpha}h(x)\left(I_{q,0^+}^{1-\alpha}y^{(1)}\right)(qx)\right]\right|\,d_qx\\
+\left|p(x)\CZ^{\alpha} h_m(x)I_{q,0^+}^{1-\alpha}y_m(x)-p(x)\CZ^{\alpha} h(x)I_{q,0^+}^{1-\alpha}y^{(1)}(x)\right|_{x=0}^{x=a}\\
+\int_{0}^{a}\left|\left[r(x)-\lambda_m^{(1)}w_{\alpha}(x)\right]y_m h_m-\left[r(x)-\lambda^{(1)}w_{\alpha}(x)\right]y^{(1)}h\right|\,d_qx.
\end{gathered}\ee
For the first $q$-integral in~\eqref{38}, by adding and subtracting the term \[D_qp(x)\CZ^{\alpha} h(x)\left(I_{q,0^+}^{1-\alpha}y_m\right)(qx)\] to the integrand, we obtain
\begin{eqnarray*}
&&\Scale[.95]{\int_{0}^{a}\left|D_q p(x)\left[\CZ^{\alpha}h_m(x)\left(I_{q,0^+}^{1-\alpha}y_m\right)(qx)-\CZ^{\alpha} h(x)\left(I_{q,0^+}^{1-\alpha}y^{(1)}\right)(qx)\right]\right|\,d_qx}\\
&\leq&\norm{D_q\,p}\norm{\CZ\,h}_{\infty}\norm{\left(I_{q,0^+}^{1-\alpha}y_m\right)(qx)-\left(I_{q,0^+}^{1-\alpha}y^{(1)}\right)(qx)}_1\\
&&+\norm{D_q\,p}M_3K_{1-\alpha}\norm{\CZ^{\alpha}(h_m-h)}_2\\
&\leq&\frac{\norm{D_q\,p}}{q}\left\{\norm{\CZ\,h}_{\infty}\norm{I_{q,0^+}^{1-\alpha}y_m-I_{q,0^+}^{1-\alpha}y^{(1)}}_1+M_3K_{1-\alpha}\norm{\CZ^{\alpha}(h_m-h)}_2\right\}
\end{eqnarray*}

where $K_{1-\alpha}$ is the constant defined in \eqref{ineq3} and $M_3:=\sup_{m\in\mathbb{N}}\norm{y_m}_{\infty}$. From \eqref{l1} and \eqref{l2}
\[\lim_{m\to\infty}\norm{y_m-y^{(1)}}=\lim_{m\to\infty}\norm{D_q h_m-D_qh}_2=0,\]
then applying \eqref{ineq1}--\eqref{ineq2}, we obtain
\[\lim_{m\to\infty}\norm{I_{q,0^+}^{1-\alpha}y_m-I_{q,0^+}^{1-\alpha}y^{(1)}}_1=\lim_{m\to\infty}\norm{\CZ^{\alpha}(h_m-h)}_2=0\] and the first $q$-integral vanishes as $m\to\infty$. As, for the second $q$-integral,  we add and subtract the term
$p(qx)D_q\CZ^{\alpha}h(x)\left(I_{q,0^+}^{1-\alpha}y_m\right)(qx)$. This gives
\begin{eqnarray*}
&&\Scale[.95]{\int_{0}^{a}\big|p(qx)\left[D_q\CZ^{\alpha}h_m(x)\left(I_{q,0^+}^{1-\alpha}y_m\right)(qx)-D_q\CZ^{\alpha}h(x)\left(I_{q,0^+}^{1-\alpha}y^{(1)}\right)(qx)\right]\big|\,d_qx}\\
&\leq& \norm{p}\norm{D_{q}\CZ^{\alpha}h}_{2}\norm{\left(I_{q,0^+}^{1-\alpha}y_m\right)(qx)-\left(I_{q,0^+}^{1-\alpha}y^{(1)}\right)(qx)}_2\\
&&+\norm{p}M_3K_{1-\alpha}\norm{D_q\CZ^{\alpha}(h_m-h)}_2\\
&\leq&\frac{\norm{p}}{q}\left\{\norm{D_{q}\CZ^{\alpha}h}_{2}\norm{I_{q,0^+}^{1-\alpha}y_m-I_{q,0^+}^{1-\alpha}y^{(1)}}_2+M_3K_{1-\alpha}\norm{D_q\CZ^{\alpha}(h_m-h)}_2\right\}.
\end{eqnarray*}
Since $D_q\CZ^{\alpha}f(x)=I_q^{1-\alpha}D_q^2 f$ if $D_qf(0)=0$, and since $\lim_{m\to\infty}\norm{D_q^2h_m-D_q^2h}_2=0$ then
from \eqref{ineq2}, the second $q$-integral tends to zero as $m$ tends to $\infty$.
For the next two terms, we have for $x=0,\, a$
\[(I_{q,0^+}^{1-\alpha}y_m)(qx)=q^{1-\alpha}I_{q,0^+}^{\alpha}y_m(qx)\to q^{1-\alpha}I_{q,0^+}^{\alpha}y^{(1)}(qx)\]
resulting from the convergence of the sequence $\norm{y_m-y}\to 0$,
and at the points $x=0,x=a$, we obtain
\[\RD h_m(0)\to \RD h(0),\; \RD h_m(a)\to \RD h(a).\]
Therefore,
\[\left|p(x)\RD h_m(x)I_{q,0^+}^{1-\alpha}y_m(x)-p(x)\RD h(x)I_{q,0^+}^{1-\alpha}y^{(1)}(x)\right|_{x=0}^{x=a}=0.\]
Similarly,  the last term in the estimation \eqref{38} vanishes as $m\to \infty$.
\vskip 0.5 cm

\noindent{\bf Step 4}
Since \be\label{EQ:5} I=\int_{0}^{a} p(x)\CZ^{\alpha}y(x)\CZ^{\alpha}h(x)+(r(x)-\lambda w_{\alpha})y(x)h(x)\, d_qx=0.\ee
Set
\bea
\gamma_1(x)&:=&p(x)\CZ^{\alpha}y(x)\non\\
\gamma_2(x)&:=&(r(x)-\lambda\,w_{\alpha})y(x)\non\\
\eea
Thus, since $h(0)=h(a)=0$, then
\[I=\int_{0}^{a}\left[\left( I_{q,a^-}^{1-\alpha}\gamma_1\right)(x)-(I_{q,0^+}\gamma_2)(qx)\right]D_qh(x)\,d_qx=0.\]
Hence, from Lemma~\ref{lem0} there is a constant $c$ such that
\be\label{PE} \left(I_{q,a^-}^{1-\alpha}\gamma_1\right)(x)-(I_{q,0^+}\gamma_2)(qx)=c,\; \forall x\in A_{q,a}^*.\ee
Acting on the two sides of \eqref{PE} by $-\frac{1}{q}D_{q^{-1}}$, we obtain
\[D_{q,a^-}^{\alpha}\gamma_1(x)+\gamma_2(x)=0,\;x\in A_{q,qa}^*.\]
Hence, $y$ is a solution of the qFSLP.
\vskip .5 cm
\noindent{\bf Step 5} In the following, we show that $(y_m^{(1)})_{m\in\mathbb{N}}$ itself converges to $y^{(1)}$. First, from Theorem~\cite[Theorem 3.12]{Zeinab_FSL1}, for a given $\lambda$ the solution of
\be\label{FSLE1}
\left[D_{q,a^-}^{\alpha} p(x)\CZ^{\alpha}y +r(x)\right]y(x)=\lambda w_{\alpha}(x)y(x),
\ee
subject to the boundary conditions
\be\label{FSLP1} y(0)=y(a)=0\ee and  the normalization condition
\be\label{FSLP2}\int_{0}^{a}w_{\alpha}(x)y^2(x)\,d_qx=1\ee
is unique except for a sign.  Let us assume that $y^{(1)}$ solves \eqref{FSLE1} and the corresponding eigenvalue is $\lambda=
\lambda^{(1)}$. Suppose that $y^{(1)}$ is nontrivial, i.e., there exists $x_0\in A_{q,qa}^*$ such that $y(x_0)\neq 0$ and choose the sign so that $y^{(1)}(x_0)>0$.
Similarly, for all $m\in\mathbb{N}$, let $y_m^{(1)}$ solve \eqref{FSLE1} with corresponding eigenvalue $\lambda=\lambda_m^{(1)}$, and let us choose the sign so that $y_m^{(1)}(x_0)\geq 0$. Now, suppose that
$(y_m^{(1)})$ does not converges to $y^{(1)}$. It means that we can find another subsequence of $y_m^{(1)}$ such that it converges to another solution $\widetilde{y}^{(1)}$.
But for $\lambda=\lambda^{(1)}$, the solution of \eqref{FSLE1}--\eqref{FSLP2} is unique except for a sign, hence
\[\widetilde{y}^{(1)}=-y^{(1)}\]
and we must have $\widetilde{y}^{(1)}(x_0)<0$. However, this is impossible because for all $m\in\mathbb{N}$, $y_m^{(1)}(x_0)\geq 0$. A contradiction, hence the solution is unique.

\vskip 0.6 cm
\noindent{\bf Step 6} In order to find eigenfunction $y^{(2)}$ and the corresponding eigenvalue $\lambda^{(2)}$, we minimize the functional \eqref{qF} subject to \eqref{qF2} and \eqref{qF1} but now with an extra orthogonality condition
\[\int_{0}^{a}y(x)y^{(1)}(x)w_{\alpha}(x)\,d_qx=0. \] If we approximate the solution by
\[y_m(x)=\frac{1}{\sqrt{w_{\alpha}}}\sum_{k=1}^m\frac{\beta_k}{\sqrt{\mu_k}}S_q(\frac{w_k x}{a}),\quad y_m(0)=y_m(a)=0,\]
then we again receive the quadratic form \eqref{C0}. However, admissible solutions are  satisfying~\eqref{C} together with
\be\label{CC}
\frac{a\sqrt{q}}{2}\sum_{k=1}^{m}\beta_k \beta_k^{(1)}=0,\ee i.e. they lay in the $(m-1)$-dimensional sphere. As before, we find that the function $\widetilde{J}([\beta])$ has a minimum
$\lambda_m^{(2)}$ and there exists $\lambda^{(2)}$ such that
\[\lambda^{(2)}=\lim_{m\to\infty} \lambda_m^{(2)},\]
because $J(y)$ is bounded from below. Moreover, it is clear that
 \be\label{FE1}
\lambda^{(1)}\leq \lambda^{(2)}.
\ee
The function $y_m^{(2)}$ defined by
\[
y_m^{(2)}(x):=\frac{1}{\sqrt{w_{\alpha}}}\sum_{k=1}^{m}\frac{\beta_k^{(2)}}{\sqrt{\mu_k}}S_q(\frac{w_k\,x}{a} ),
\]
   achieves its minimum $\lambda_m^{(2)}$, where $\beta^{(2)}=\left(\beta_1^{(2)},\ldots,\beta_m^{(2)}\right)$ is the point satisfying \eqref{C}
  and \eqref{CC}. By the same argument as before, we can prove that the sequence $(y_m^{(2)})$ converges uniformly  to a limit function $y^{(2)}$, which satisfies the
  qFSLP  \eqref{EVP3} with $\lambda^{(2)}$, boundary conditions \eqref{qF2} and orthogonality condition \eqref{qF1}. Therefore, solution $y^{(2)}$
  of the qFSLP corresponding to the eigenvalue $\lambda^{(2)}$ exists. Furthermore, because orthogonal functions cannot be linearly dependent, and since only one eigenfunction corresponds to each eigenvalue(except for a constant  factor)
  we have the strict inequality \[\lambda^{(1)}<\lambda^{(2)}\]instead of \eqref{FE1}. Finally, if we repeat the above procedure with similar modifications, we can obtain eigenvalues
  $\lambda^{(3)},\,\lambda^{(4)},\ldots$ and corresponding eigenvectors $y^{(3)}, y^{(4)},\ldots$.
\end{proof}
\subsection{The first eigenvalue}
\begin{defn}
The Rayleigh quotient for the $q$ fractional Sturm--Liouville problem \eqref{EVP3}--\eqref{BC6} is defined by
\[R(y):=\dfrac{J(y)}{I(y)},\]
where $J(y)$ and $I(y)$ are given by \eqref{qF} and \eqref{qF1}, respectively.
\end{defn}

\begin{thm}Let   $y$ be a non zero function satisfying $y$ and $\CZ^{\alpha}y$  are in $C(A_{q,a}*)$   and $y(0)=y(a)=0$. Then,  $y$ is a  minimizer of $R(y)$ and $R(y)=\lambda$
if and only if $y$ is  an eigenfunction  of problem \eqref{EVP3}--\eqref{BC6} associated with  $\lambda$.  That is, the minimum
value of $R$ at $y$ is the first
 eigenvalue $\lambda^{(1)}$.
\end{thm}
\begin{proof}First, we prove the necessity.
 Assume that $y$ is a non zero minimizer of $R(y)$ and $R(y)=\lambda$.
Consider the one parameter family of curves
\[y=y+h\eta,\quad |h|\leq \epsilon,\]
where $\eta$ and $\CZ^{\alpha} $ are $C(A_{q,a}^*)$ functions and $\eta(0)=\eta(a)=0$ and $\eta\neq 0$.
Define  functions  $\phi,\psi,\xi$ on $[-\epsilon,\epsilon]$ by
\[\phi(h):=I(y+ h\eta),\; \psi(h):=J(y+h\eta),\; \xi(h)=R(y+h\eta)=\frac{\psi(h)}{\phi(h)},\;h\in[-\epsilon,\epsilon].\]
Hence $\xi$ is $C^1$ function on  $[-\epsilon,\epsilon]$. Since $\xi(0)=R(y)$, then $0$ is a minimum value of $\xi$. Consequently,
$\xi_i'(0)=0$. But
\[\xi'(h)=\frac{1}{\phi(h)}\left(\psi'(h)-\frac{\psi(h)}{\phi(h)}\phi'(h)\right)\]
and
\begin{eqnarray*}\psi'(0)&=&2\int_{0}^{a}\left[p(x)\CZ^{\alpha}y\,\, \CZ^{\alpha}\eta+r(x)y\eta\right]\,d_qx,\\
\phi'(0)&=&2\int_{0}^{a}w_{\alpha}(x)y(x)\eta(x)\,d_qx,\\
\frac{\psi(0)}{\phi(0)}&=&R(y)=\lambda.
\end{eqnarray*}
Therefore,
\[\xi'(0)=\frac{2}{I(y)}\left(\int_{0}^{a}\left[p(x)\CZ^{\alpha}y \CZ^{\alpha}\eta+\left(r(x)y-\lambda w_{\alpha}\right)\eta\right)\,d_qx\right).\]
Using \eqref{CP3}, we obtain
\[\int_{0}^{a}\left[D_{q,a^-}^{\alpha}P(x)\CZ^{\alpha}y(x)+\left(r(x)-\lambda\right) w_{\alpha}(x)y(x)\right]\eta(x)\,d_qx=0.\]
Applying~Lemma~\ref{lem4}, we obtain
\[D_{q,a^-}^{\alpha}P(x)\CZ^{\alpha}y(x)+r(x)y(x)=\lambda w_{\alpha}(x)y(x),\quad x\in A_{q,qa}^*.\]
This proves the necessity. Now we prove the sufficiency.
Assume that $y$ is an eigenfunction of \eqref{EVP3}--\eqref{BC6} associated with an eigenvalue $\lambda$. Then
\be\label{Eq:0}
D_{q,a^-}^{\alpha}P(x)\CZ^{\alpha}y(x)+r(x)y(x)=\lambda w_{\alpha}(x)y(x),\quad x\in A_{q,qa}^*.
\ee
Multiply \eqref{Eq:0} by $y$ and calculate the $q$-integration from 0 to $a$, we obtain
\[
\int_{0}^{a} \left[y(x)D_{q,a^-}^{\alpha}P(x)\CZ^{\alpha}y(x)+r(x)y^2(x)\right]\,d_qx=\lambda \int_{0}^{a}w_{\alpha}(x)y^2(x)\,d_qx.
\]
Since $y\neq 0$, then $\int_{0}^{a}w_{\alpha}(x)y^2(x)\,d_qx>0$, and
\[\dfrac{\int_{0}^{a} \left[y(x)D_{q,a^-}^{\alpha}P(x)\CZ^{\alpha}y(x)+r(x)y^2(x)\right]\,d_qx}{\int_{0}^{a}w_{\alpha}(x)y^2(x)\,d_qx}=\lambda.\]
I.e. $R(y)=\lambda$.
 Therefore, any minimum value of $J$ is an eigenvalue and it is attained at the associated eigenfunction. Therefore the minimum value of $J$ is the smallest
  eigenvalue.

\end{proof}

\section{Conclusion and future work}
  This paper is the first paper deals with variational problems of functionals defined in terms of Jackson $q$-integral on finite domain and the left sided Caputo $q$-derivative appears in the integrand. We give a fractional $q$-analogue of the Euler--Lagrange equation and a $q$-isoperimetric problem is defined and solved. We use these results in recasting the qFSLP under consideration as a $q$-isoperimetric problem, and then we solve it by a technique similar to the one used in solving  regular Sturm--Liouville problems in~\cite{GFbook} and fractional Sturm--Liouville problems in~\cite{KOM}. This complete the work started by the author in~\cite{Zeinab_FSL1}, and  generalizes the study of integer Sturm--Liouville problem introduced by Annaby and Mansour in~\cite{AM1}. A similar study for the fractional
Sturm--Liouville problem
\[{}^cD_{q,a^-}p(x)D_{q,0^+}^{\alpha}y(x)+r(x)y(x)=\lambda w_{\alpha}(x)y(x),\]
is in progress.

 \end{document}